\documentclass{amsart}
\usepackage{amsfonts,amssymb,amsmath,amsthm}
\usepackage{url,color}
\usepackage{enumerate}
\usepackage[comma, authoryear]{natbib}
\newtheorem{thm}{Theorem}[section]
\newtheorem{lemma}[thm]{Lemma}
\newtheorem{prop}[thm]{Proposition}

\theoremstyle{definition}
\newtheorem{defn}[thm]{Definition}
\newtheorem{rem}[thm]{Remark}
\numberwithin{equation}{section}

\def\H{\mathcal{H}}
\def\M{\mathbb{M}}
\def\R{\mathbb{R}}
\def\F{\mathbb{F}}
\def\disp{\displaystyle}
\def\bnu{\boldsymbol\nu}
\def\bmu{\boldsymbol\mu}

\begin{document}

\title[Solutions of stochastic evolution equations]{On signed measure valued solutions of stochastic evolution equations}

\author{Bruno R\'{e}millard}
\address{GERAD and Ser\-vi\-ce
d'en\-sei\-gne\-ment des m\'{e}\-tho\-des quan\-ti\-ta\-ti\-ves de
ges\-tion, HEC Montr\'{e}al,\\
3000, che\-min de la
C\^{o}\-te-Sain\-te-Ca\-the\-ri\-ne, Montr\'{e}al (Qu\'{e}\-bec), Canada H3T
2A7}
\email{bruno.remillard@hec.ca}
\thanks{Partial funding in support of this work was provided by
the Natural Sciences and Engineering Research Council of Canada, and
by the Fonds qu\'{e}b\'{e}cois de la recherche sur la nature et les
technologies}
\author{Jean Vaillancourt}
\address{Rectorat,
UQO, 283, boulevard Alexandre-Tach\'{e},\\
Case postale 1250, succursale Hull Gatineau (Qu\'{e}bec),  Canada J8X 3X7}
\email{jean.vaillancourt@uqo.ca}
\thanks{Partial funding in support of this work was provided by the Natural
Sciences and Engineering Research Council of Canada. To appear in
Stochastic Processes and their Applications}

\begin{abstract}
We study existence, uniqueness and mass conservation of signed
measure valued solutions of a class of stochastic evolution
equations with respect to the Wiener sheet, including as particular
cases the stochastic versions of the regularized two-dimensional
Navier-Stokes equations in vorticity form introduced by Kotelenez.
\end{abstract}

\subjclass{Primary  76D06 60G57 60H15, Secondary 60H05 60J60}

\keywords{Signed measure; stochastic evolution equation; McKean-Vlasov; Wiener sheet}

\maketitle

\section{Introduction}
Measure-valued stochastic processes arise as mathematical
descriptors for the limiting behaviour of many characteristic
parameters used for modelling complex evolutions in the natural
sciences. Genetic drift, bacterial spread, fluid dynamics, heat
conduction and chemical reactions are but a few examples of
challenging problems for which a good mathematical understanding can
be reached by first building an appropriate interacting particle
system and then looking at it through the lens of the associated
(measure valued) empirical process. When properly rescaled, the
processes give rise to measure valued limits solving stochastic
evolution equations which are themselves of great interest. The
study of the existence and uniqueness of measure valued solutions to
stochastic evolution equations really started with
\cite{Dawson:1975} and the subject has grown rapidly since then. For
a nice review of
measure valued processes, see \cite{Dawson:1993}.\\

A broad class of such equations, the so-called stochastic
McKean-Vlasov equations, e.g., \cite{Dawson/Vaillancourt:1995}, are
probability valued if their starting point is. This conservation of
the initial mass turns out to be quite easy to show in this case, by
way of a basic completeness argument which remains valid for many
other examples in the literature. Caution must prevail, however,
against a common misconception which recurs unfortunately too often,
to the effect that an interacting particle system displaying no
creation or annihilation of particles at any time, necessarily gives
rise in its scaling limits to stochastic
evolutions obeying this mass conservation property. While seemingly sensible, such a statement requires proof.\\

For a general treatment of the important and challenging family of stochastic Navier-Stokes equations in vorticity form, such as those
appearing in \cite{Kotelenez:1995}, it is necessary to look for solutions in the space of
signed measures. Mainly because the space of signed measures is not complete for the usual metrics compatible
with the topology of weak convergence, it is much harder to study existence, uniqueness and/or
mass conservation of signed measure valued solutions of stochastic evolution equations.\\

In fact, because of the incompleteness, many of the results
presented in the literature on signed measure valued solutions are
either false or have only been provided with a false proof. This is
the case for example in \cite{Marchioro/Pulvirenti:1982},
\cite{Kotelenez:1995} and \cite{Amirdjanova:2000}, where the space
of signed measures was assumed to be complete.
Other examples of incorrect proofs include  \citet{Kotelenez:2010}, \citet{Kotelenez/Seadler:2011,Kotelenez/Seadler:2012}.
These articles will be discussed later. \\

In this paper, we study existence and uniqueness of signed measure
valued solutions of a class of stochastic evolution equations with
respect to the Wiener sheet, including as particular cases the
stochastic versions of the regularized two-dimensional Navier-Stokes
equations in vorticity form introduced by Kotelenez.  These
stochastic evolution equations can be seen as weak versions of
equations of the form
$$
d \chi_t = L^*_{\chi_t} \chi_t -\nabla (\Gamma \chi_t) dW,
$$
where $L^*$ is the (formal) adjoint of a diffusion operator with
coefficients depending on $\chi$ and $W$ is the Wiener sheet. Other
weak versions of these equations appear for example in
\citet{Amirdjanova/Xiong:2006}. Also, additional references to
similar equations are given in Section 5 where we discuss the
relationship with two-dimensional regularized Navier-Stokes
equations in vorticity form.
\\

The problem  is described in Section \ref{sec:description}, where
appropriate spaces of measures are defined. Next, in Section
\ref{sec:existence}, using a particles representation, signed
measure valued solutions are shown to exists when the initial signed
measure has finite support. The existence is then shown to hold for
general initial conditions. Uniqueness and mass conservation are
shown to hold in Section \ref{sec:mappingS}, using fixed points
arguments and duality. Finally, in Section \ref{sec:relation}, we
revisit some of the results appearing in the literature concerning
two-dimensional Navier-Stokes equations in vorticity form.\\

A commendable attempt at resolving all these issues can be found in
\citet{Kurtz/Xiong:1999},  where the authors show existence of
solutions to a large class of nonlinear stochastic partial
differential equations encompassing our own. Their techniques also
allow them to prove uniqueness of solution, but only for those
starting measures with square integrable densities with respect to
Lebesgue measure. Our constructions yield the existence and
uniqueness of solution for all starting signed measures.

\section{Description of the problem}\label{sec:description}

First, one needs to define correctly the spaces of measures and
signed measures that  will be used in the sequel. To this end,
suppose that $\rho$ is the metric on $\R^d$ defined by
$$
\rho(x,y) = \min(1,\|x-y\|), \quad x,y\in \R^d,
$$
where $\|\cdot\|$ is the Euclidean norm.\\

Suppose that $p\ge 1$ is given. For any  $m\ge 0$, let $M(m)$ be the
set of (non negative) Borel measures $\mu$ with $\mu(\R^d) =m$. This
space is equipped with the Wasserstein metric $W_p$ defined by
$$
W_p(\mu,\nu) = \left[ \inf_{Q\in \H(\mu,\nu)}\int_{\R^{2d}}
\rho^p(x,y) Q(dx,dy)\right]^{1/p},
$$
where $\H(\mu,\nu)$ is the set of all joint representations of
$(\mu,\nu)$, that is, the set of all Borel measures $Q$ on $\mathbb{R}^{d}\times \mathbb{R}^{d}$, so that
for any Borel subset $A$ of $\mathbb{R}^{d}$,
$Q(A\times \R^d) = m\mu(A)$ and $Q(\R^d \times A) = m\nu(A)$.\\

Further set $M(m_1 ,\,m_2)=  M(m_1)\times M(m_2)$ and consider the
following ``Wasserstein'' metric $\gamma_p$ on $M(m_1 ,\,m_2)$:
$$
\gamma_p(\boldsymbol\mu,\boldsymbol\eta) = \gamma_p \left\{ (\mu^1,\mu^2),(\eta^1,
\eta^2)\right\} = \left[ W_p^p \left( \mu^1 ,\eta^1\right) + W_p^p
\left( \mu^2 ,\eta^2\right)\right]^{1/p}.
$$
It follows easily from Minkowski's inequality together with a lemma
in \cite[p.330 ]{Dudley:1989} that for any $p\ge 1$, $\gamma_p$ is a
metric on $M(m_1 ,\,m_2)$, since $W_p$ is a metric on $M(m)$.\\

Let $M^{(f)}(m)$ be the subset of
measures in $M(m)$ concentrated on finite sets and $M^{(f)}(m_1,m_2)
= M^{(f)}(m_1)\times M^{(f)}(m_2)$.\\

Notice that since $\rho \le 1$, for any $\mu,\nu \in M(m)$, one has
$$
W_p^p(\mu, \nu) \le W_1(\mu, \nu) \le  (m^2)^{\frac{p-1}{p}}
W_p(\mu, \nu).
$$
It follows that for any $\boldsymbol\mu,\boldsymbol\nu \in M(m_1,m_2)$,
$$
\gamma_p^p(\boldsymbol\mu, \boldsymbol\nu) \le \gamma_1(\boldsymbol\mu, \boldsymbol\nu) \le (2
m^2)^{\frac{p-1}{p}} \gamma_p(\boldsymbol\mu, \boldsymbol\nu),
$$
where $m= \max(m_1,m_2)$. It follows that $\gamma_p$ and $\gamma_2$
generate the same topology. Therefore, from now on, we use
$p=2$.\\

The following lemma summarizes results in \citet{Huber:1981} and
\citet{Dudley:1989}[Lemma 11.8.4] about spaces of measures endowed
with metrics $W_2$ and $\gamma_2$.

\begin{lemma}\label{lem:polish}
$(M(m),W_2)$ and $(M(m_1 ,\,m_2),\gamma_2)$ are Polish spaces. Their
respective topology is that of weak convergence on $M(m)$ and
$M(m_1,m_2)$.  Moreover  $M^{(f)}(m)$ is dense in $M(m)$ and
$M^{(f)}(m_1,m_2)$ is dense in $M(m_1,m_2)$.
\end{lemma}

Throughout the rest of the paper, let $m_1$, $m_2$ be fixed.\\

Set $M = M(m_1, m_2)$ and $M^{(f)} =  M^{(f)}(m_1, m_2)$. Further
let $\mathcal{M}$ be the space of $M$-valued random variables with
metric $\gamma$ defined by
$$
\gamma(\boldsymbol\chi,\boldsymbol\eta) = \left[ E
\gamma_2^2(\boldsymbol\chi,\boldsymbol\eta)\right]^{1/2},\quad \boldsymbol\chi,\boldsymbol\eta\in \mathcal{M},
$$
and denote by $\mathcal{M}^{(f)}$ the subset of $M^{(f)}$-valued
random variables.\\

Let $C([0,T];M)$ denote the space of continuous mappings from $[0,T]$ into $M$ and let $\mathbb{M}$ be the space of
$C([0,T];M)$-valued random variables with metric $\gamma_{[0,T]}$
defined by
$$
\gamma_{[0,T]}(\boldsymbol\chi,\boldsymbol\eta ) = \left[ E \sup_{0\leq t\leq
T}\gamma_{2}^{2}(\boldsymbol\chi_t,\boldsymbol\eta_t )\right]^{1/2},\quad \boldsymbol\chi,\boldsymbol\eta\in
\mathbb{M},
$$
and let $\mathbb{M}^{(f)}$  be the subset of
$C([0,T];M^{(f)})$-valued random variables.\\

 It follows easily that $(\mathcal{M},\gamma)$ is
a complete metric space, since $(M,\gamma_2)$ is complete, by Lemma
\ref{lem:polish}. Note that $\mathcal{M}^{(f)}$ is dense in
$\mathcal{M}$. Furthermore, $(\M,\gamma_{[0,T]})$ is also complete
and  $\M^{(f)}$ is a dense subset of $\mathbb{M}$.\\

Finally, let $SM(m_1,m_2)$ be the set of Borel signed measures with
Hahn-Jordan decomposition $(\mu^1,\mu^2)\in M(m_1,m_2)$. A natural
distance on $SM(m_1,m_2)$ is the one inherited from $\gamma_2$. More
precisely, if $\boldsymbol\mu=(\mu^1,\mu^2)$ and $\boldsymbol\nu=(\nu^1,\nu^2)$ are respectively
the Hahn-Jordan  decompositions of
$\mu,\nu\in SM(m_1,m_2)$, then the distance between $\mu$ and $\nu$ is $\gamma_2(\boldsymbol\mu, \boldsymbol\nu)$. \\

\begin{rem}
Many authors, because of Lemma \ref{lem:polish}, have stated that
$SM(m_1, m_2)$, endowed with that metric, is complete, while it is not.\\

To see that, simply take $\mu_n = \delta_{1/n}-\delta_{0}$. Then
clearly its Hahn-Jordan decomposition is
$(\delta_{1/n},\delta_{0})$, so $\mu_n\in SM(1,1)$. Moreover, it is
obvious that $\mu_n$ is a Cauchy sequence in $SM(1,1)$ since
$\gamma_2((\delta_{1/n},\delta_0),(\delta_{1/m},\delta_0))\to 0$ as
$n, m\to\infty$. However, there is no measure $\mu\in SM(1,1)$ with
Hahn-Jordan decomposition $(\mu_1,\mu_2)$ so that
$\gamma_2((\delta_{1/n},\delta_0),(\mu_1,\mu_2))\to 0$ as
$n\to\infty$. For, the latter implies that $\mu_1=\mu_2=\delta_0$,
so
$\mu = \delta_0-\delta_0=0 \not\in SM(1,1)$.\\
\end{rem}

In order to preserve the mass in the limit,
\citet{Kotelenez/Seadler:2011,Kotelenez/Seadler:2012}  introduce the
following metric $\lambda$ on the space of signed measures: If $\mu$
and $\nu$ have Hahn-Jordan decompositions $(\mu^+,\mu^-)$ and
$(\nu^+,\nu^-)$, then the distance between $\mu$ and $\nu$, denoted
by $\lambda(\mu,\nu)$, is defined by
\begin{eqnarray}\label{eq:KSdistance}
\lambda(\mu,\nu) &=&
\inf_{\eta\in M(m), m\ge 0} \max\left[
\gamma(\mu^+ -\nu^+ -\eta)+\gamma(\mu^- -\nu^- -\eta), \right.\\
&& \qquad \qquad \qquad \qquad \qquad \left. \gamma(\mu^+ -\nu^+ +\eta)+\gamma(\mu^- -\nu^- +\eta)\right],\nonumber
\end{eqnarray}
where $\gamma(\chi)  = \sup_{f;\|f\|\wedge \|f\|_L\le 1}|<\chi,f>|$,
for any finite signed measure $\chi$,  and $\|\cdot\|_L$ is the
so-called minimal Lipschitz constant defined by
\eqref{eq:lipschitz}.

It is shown in \citet[Theorem A.6]{Kotelenez/Seadler:2012} that with
respect to $\lambda$,  a Cauchy sequence $\mu_n$ converges to $\mu$
if and only if there  exists a sub-sequence $\mu_{n_k}$ so that the
Hahn-Jordan decompositions $\mu_{n_k}^\pm $ converges to the
Hahn-Jordan decomposition $\mu^\pm$ of $\mu$.

\begin{rem}\label{rem:KSdistance}
This new metric is not as useful as it seems since,  in order to
show that a sequence $\mu_n$ converges to a limit $\mu$, one needs
to prove the convergence of the Hahn-Jordan decompositions to that
limit, at least for a subsequence. In
\citet{Kotelenez/Seadler:2011,Kotelenez/Seadler:2012}, from the
convergence of the Hahn-Jordan decompositions $\mu_n^+$ and
$\mu_n^-$ to measures $\mu_1$ and $\mu_2$, they conclude that the
sequence is a Cauchy sequence in $\lambda$, which is true, but they
also conclude that the sequence converges in $\lambda$, which is not
necessarily true. To see that, take the same example as before,
i.e., $\mu_n = \delta_{1/n}-\delta_0$. Since $\mu_n^\pm \to
\delta_0$,  the sequence is Cauchy with respect to $\lambda$.
However it is not a convergent sequence, according to \citet[Theorem
A.6]{Kotelenez/Seadler:2012}.
\end{rem}

\begin{rem}\label{rem:hypotheses}
Before defining the stochastic evolution equations of interest here,
we state some properties which are assumed to hold throughout the
rest of this paper.

Suppose that $K(x,y):\R^d\times \R^d \mapsto \R^d $ is bounded, differentiable with respect to $x$ and
Lipschitz in both variables, i.e there exists a constant $C$ such that
$$
\|K(x,y)-K(x',y')\| \le C\left\{ \|x-x'\|^2+
\|y-y'\|^2\right\}^{1/2}
$$
holds for all $x,x', y,y' \in \R^d$.

Moreover assume that $\Gamma : \R^d\times \R^d \mapsto \R^{d\times
d} $ is a continuous function that satisfies
\begin{equation}\label{eq:gamma}
\sum_{j=1}^d \sum_{l=1}^d \int \left\{
\Gamma_{jl}(r,p)-\Gamma_{jl}(q,p)\right\}  ^{2}dp\leq C_\Gamma^2
\|r-q\|^2.
\end{equation}
for some positive constant $C_\Gamma$ and $G(x,y) = \int_{\R^d} \Gamma(x,p)\Gamma(y,p)^\top dp$ is such that
$$
\max_{1\le i,j\le d} \sup_{x,y} \left\{ \frac{|G_{ij}(x,y)|}{(1+|x|)(1+|y|)}+  \max_{1\le k, l\le d}\left|\partial_{x_k}\partial_{y_l}G_{ij}(x,y)\right|\right\} < \infty.
$$
\end{rem}

Given is a stochastic basis $(\Omega,\mathcal{ F},\mathbb{F} =
(\mathcal{ F}_t)_{t\ge 0},P)$ which satisfies the usual conditions.
All stochastic processes are assumed to live on $\Omega$ and to be
$\F$-adapted, including all initial conditions in SDE's and SPDE's.
Moreover, the processes are assumed to be $P\otimes \lambda$
measurable, where $\lambda$ is the Lebesgue measure on
$\left[  0,\infty\right)$.

Recall that a Wiener sheet $w$ is a stochastic process defined on
$\R^d \times[0,\infty)$, such that for any Borel set $A$ with finite
Lebesgue measure $\lambda(A)$,  $t\mapsto w(A,t)$ is a continuous
centered Gaussian process with covariance function
$E\{w(A,s)w(B,t)\} = \min(s,t)\lambda(A\cap B)$, whenever $A$ and $B$ both have finite
Lebesgue measures. That implies, in particular that $w(A\cup B,t) =
w(A,t)+w(B,t)$ almost surely, whenever $A$ and $B$ both have finite
Lebesgue measures and are disjoint.\\

The following stochastic evolution equation, with respect to a
vector $w=(w_1,\ldots, w_d)$ of independent Wiener sheets,  will be analyzed next.
To be adapted  for $w_{l}(r,t)$ means that $\disp \int_{A}%
$~$w_{l}(dp,t)$ is adapted for any Borel set $A\subset {\mathbb
R}^d$ with\ $\lambda(A) <\infty$.
The space of twice-differentiable real-valued functions with bounded second derivative on ${\mathbb R}^d$ is henceforth denoted
by $C_{b}^{2}({\mathbb R}^d)$.\\\\

\begin{defn}\label{def:NS}
A path $\chi  = \chi^1 -\chi^2 $, with $(\chi^1,
\chi^2)\in \M$, is called a solution of the stochastic evolution
equation if
\begin{equation}\label{eq:WSNSE}
d<\!\chi_t,f\!> = <\!\chi_t, L(\chi_t) f \!>dt + \int <\!\chi_t, \nabla f(\cdot)^\top\Gamma(\cdot
,p)\!>w(dp,dt),
\end{equation}
holds for all $f\in C_{b}^{2}({\mathbb R}^d)$, where we write, for any $x\in\mathbb{R}^d$,
$$
U (x,\chi_t) = \int K(x,q)\chi_t(dq),
$$
and
$$
L (\chi_t)f(x) = \sum_{j=1}^{d} \partial_{x_j}f(x)U_j
(x,\chi_t) +\frac{1}{2} \sum_{j=1}^{d} \sum_{k=1}^{d}
\partial_{ x_j}\partial_{x_k} f(x)G_{jk}(x,x),
$$
with
$<\!\chi,f\!>$ being the integral of $f$ with respect to $\chi$.\\
\end{defn}

The existence of such a solution is studied in the next section.\\

\begin{rem}\label{rem:conservation}
Note that even if $\chi_t = \chi_t^1-\chi_t^2$, with
$( \chi_t^1,\chi_t^2) \in \mathcal{M}$ for all $t$, it does
not necessarily imply that $( \chi_t^1,\chi_t^2)$ is the Hahn-Jordan decomposition
of $\chi_t$. That (desirable) property is hereafter called
``mass conservation''. It will be studied in Section
\ref{sec:mappingS}. In the setting of \cite{Kotelenez:1995}, this is
called conservation of vorticity and this has important physical
implications. See, e.g., \cite{Chorin/Marsden:1993}.
\end{rem}

\section{Existence of signed measure valued solutions}\label{sec:existence}

The key argument in proving the existence of a solution to
\eqref{eq:WSNSE} is the construction of an
appropriate particles system.\\

For instance, consider the following system of SODEs that describe
the movement of $N$ interacting particles $r^1(t),\ldots, r^N(t)$:
\begin{equation}\label{eq:sys-part}
\left\{
\begin{array}
[c]{l}%
dr^{i}(t)=\sum_{j=1}^{N}a_{j}K(r^{i}(t),r^{j}(t))dt+ \int \Gamma(r^{i}(t),p)w(dp,dt), \\
r^{i}(0)=r_{0}^{i},\; i=1,2,\ldots   , N,
\end{array}
\right.
\end{equation}
where $a_1,\ldots, a_N$ are fixed real numbers.\\

Existence and uniqueness of system of particles \eqref{eq:sys-part}
was stated in \citet{Kotelenez:1995} for particular cases of $K$ and
$\Gamma$. Under the general Lipschitz conditions stated in
Remark \ref{rem:hypotheses},
one has the following result.

\begin{lemma}\label{lem:sys-sde}
For every $r_0\in \R^{dN}$, \eqref{eq:sys-part} has a unique
$\F_{t}$-adapted solution $r\in C(\left[ 0,\infty\right) ;{\mathbb
R}^{dN})$ a.s., which is an ${\mathbb R}^{dN}$-valued Markov
process.\
\end{lemma}

For the sake of completeness, the proof is given in Appendix
\ref{app:kol},  completing some missing arguments in
\citet{Kotelenez:1995}. Note that by construction, if $\pi$ is any
permutation of $\{1,\ldots, N\}$, the solution corresponding to
$a_\pi$ is $r^\pi$.

\begin{lemma}\label{lem:map}
Suppose that $(\nu^1,\nu^2) \in M^{(f)}$ is such that
$$
\nu^1 =  \sum_{i=1}^N  \max(a_i,0) \delta_{r_0^i},\qquad \nu^2 =
\sum_{i=1}^N  \max(-a_i,0) \delta_{r_0^i},
$$
for some $r_0 \in \R^{dN}$ and $a = (a_1,\ldots, a_N)^\top \in \R^N$ with the property that
$$
\sum_{i=1}^N \max(a_{i},0) = m_1,\qquad \sum_{i=1}^N \max(-a_{i},0)
= m_2.
$$

Then, the mapping $\Psi: (\nu^1,\nu^2) \in M^{(f)} \mapsto \Psi (\nu^1,\nu^2) = (\chi^1,\chi^2)
\in \M^{(f)}$, given for all $t\ge 0$ by
$$
\chi^1_t = \sum_{i=1}^N  \max(a_i,0) \delta_{r^i(t)}, \quad \chi^2_t
= \sum_{i=1}^N  \max(-a_i,0) \delta_{r^i(t)},
$$
is well-defined, where $r$ satisfies \eqref{eq:sys-part}.

Moreover,  $(\chi^1,\chi^2)\in \M^{(f)}$ gives rise
to a solution of the stochastic evolution equation, i.e., the
empirical signed measure $\chi_t$ defined by
$$
\chi_t = \chi^1_t-\chi^2_t = \sum_{i=1}^N  a_i
\delta_{r^i(t)},
$$
satisfies  \eqref{eq:WSNSE}.
\end{lemma}

\begin{proof} First, let $\pi$ be any permutation of $\{1,\ldots, N\}$.
Then  $(\chi_0^1,\chi_0^2) = (\nu^1,\nu^2)$ can also be written as
$$
\chi_0^1  =  \sum_{i=1}^N  \max(a_{\pi_i},0) \delta_{r_0^{\pi_i}},
\quad \chi_0^2  =  \sum_{i=1}^N  \max(-a_{\pi_i},0)
\delta_{r_0^{\pi_i}}.
$$
By Lemma \ref{lem:sys-sde},  it follows that the unique solution
$q$ to
$$
\left\{
\begin{array}
[c]{l}%
dq^{i}(t)=\sum_{j=1}^{N}a_{\pi_j}K(q^{i}(t),q^{j}(t))dt+ \int \Gamma(q^{i}(t),p)w(dp,dt), \\
q^{i}(0)=r_{0}^{\pi_i},\; i=1,\ldots   , N,
\end{array}
\right.
$$
is $q^i = r^{\pi_i}$. Hence $(\chi^1,\chi^2) = \Psi(\nu^1,\nu^2)$ is well-defined.\\

Let
$\chi = \chi^1 - \chi^2$ and set
$U (x,\chi_t) = \int K(x,p)\chi_t(dp)$, $x\in \mathbb{R}^d$.
Note that for all $i=1,\ldots, N$,
$$
r^i(t) = r_0^i+ \int_0^t U(r^i(s),\chi_s)ds +M^i(t),
$$
where the $\mathbb{R}^d$-valued martingale $M^i$ has components
$$
M_j^i(t) =\sum_{l=1}^d \int_0^t \int
\Gamma_{jl}(r^i(s),p)w_l(dp,ds)
$$
and quadratic covariation
\begin{equation}\label{quadvar2}
\left < \left < M_j^i, M_k^i\right> \right>(t) = \sum_{l=1}^d
\int_{0}^{t}\int \Gamma_{jl}(r^{i}(s),p) \Gamma_{kl}
(r^{i}(s),p)dpds,
\end{equation}
for any $j,k=1,\ldots, d$.\\

Now, let $f\in C_{b}^{2}({\mathbb R}^d)$. Applying
It\^{o}'s formula to
$<\chi_t,f>=\sum_{j=1}^{N}a_{j}f(r^{j}(t))$, we obtain
\begin{eqnarray*}
d<\chi_t,f> &=& \sum_{i=1}^{N}\sum_{j=1}^{d}a_{i} \partial_{x_j}f(r^{i}(t))U_j (r^{i}(t),\chi_s)dt \\
&& \quad + \frac{1}{2}\sum_{i=1}^{N}\sum_{j=1}^{d} \sum_{k=1}^{d} a_{i} \partial_{ x_j}\partial_{x_k} f(r^{i}(t))G_{jk}(r^i(t)) dt\\
&& \qquad +\sum_{i=1}^{N}\sum_{j=1}^{d} \sum_{l=1}^d a_{i} \partial_{x_j} f(r^{i}(t))\int \Gamma_{jl} (r^{i}(t),p)w_{l}(dp,dt)\\
&=& <\chi_t, L(\chi_t) f >dt \\
&& \qquad +  \int <\chi_t, \nabla f(\cdot)^\top\Gamma(\cdot
,p)> w(dp,dt),
\end{eqnarray*}
which is exactly  \eqref{eq:WSNSE}. \end{proof}

The following lemma will allow us to extend the solution of \eqref{eq:WSNSE}
for discrete initial conditions to one with arbitrary initial conditions in
$M$.

\begin{lemma}\label{lem:extension}
For an any $T>0$, there exist $c'=c'(T), c''=c''(T) >0$, independent
of $\boldsymbol\chi_0=(\chi_0^1,\chi_0^2), \boldsymbol\eta_0 =(\eta_0^1,\eta_0^2)\in M^{(f)}$, such that if $\boldsymbol\chi = \Psi(\boldsymbol\chi_0)$
and $\boldsymbol\eta = \Psi(\boldsymbol\eta_0)$, then
\begin{equation}\label{eq:relation}
\gamma_{[0,T]}(\boldsymbol\chi,\boldsymbol\eta)\leq c' \gamma_{2}(\boldsymbol\chi_0,\boldsymbol\eta_0).
\end{equation}

Moreover
\begin{equation}\label{eq:kr3}
E\sup_{0\le t\le T} \sup_{{\|f\|}_L \le 1}
<\chi_t-\eta_t,f>^2 \; \le c'' \gamma_2^2
(\boldsymbol\chi_0,\boldsymbol\eta_0).
\end{equation}
\end{lemma}

The proof is relegated to Appendix \ref{app:extension}.\\

Using the previous result, it is now possible to extend the mapping
$\Psi$ from $M^{(f)}$ to $M$. While the representation of the
mapping $\Psi$ is explicit when $\boldsymbol\nu\in M^{(f)}$, it is no longer the case
when $\boldsymbol\nu \in M\setminus M^{(f)}$.

\begin{thm}\label{thm:extpsi}
The map $\Psi :=\boldsymbol\chi = (\chi_0^1,\chi_0^2) \mapsto \boldsymbol\chi=(\chi^1,\chi^2)$ from $M_{f}$ into $\M^{(f)}$
extends uniquely to a map from $M$  into $\M$ . Moreover, for any
$\boldsymbol\chi_0,\boldsymbol\eta_{0}\in M$,
$$
\gamma_{[0,T]}(\boldsymbol\chi,\boldsymbol\eta) \le  c' \gamma(\boldsymbol\chi_0,\boldsymbol\eta_0)
$$
and $\chi = \chi^1-\chi^2$ satisfies \eqref{eq:WSNSE}, that
is $\chi$ is a solution of the stochastic evolution equation
with initial condition $\chi_0$.
\end{thm}

The proof is relegated to Appendix \ref{app:pfthmpsi}.\\

\begin{rem}\label{rem:psi}
While the last theorem tells us that there is at least one solution
to the weak stochastic evolution equation starting from some signed
measure $\nu$, one cannot deduce yet that there is a unique
solution, nor that the mass is conserved, that is, if $\chi_0
\in SM(m_1,m_2)$, then $\chi_t$ also belongs to
$SM(m_1,m_2)$ for all $t\ge 0$. In order to prove these results, one
needs to introduce another mapping.
\end{rem}

\section{Fixed point representation and mass conservation}\label{sec:mappingS}

The plan is the following. First, we start by defining a mapping
that maintains, through time, the Hahn-Jordan decomposition of the
initial signed measure. Then, we show that it has  a unique fixed
point and that the latter satisfies \eqref{eq:WSNSE}. Finally, we
prove the uniqueness of the solution of \eqref{eq:WSNSE}, using
a monotonicity condition.

\begin{defn}\label{def:SMapping}

Let $\boldsymbol\nu = (\nu^1,\nu^2) \in M$ be given. Consider the operator $S=(S^1,S^2)$
acting on $\M_{\boldsymbol\nu} = \{\boldsymbol\mu \in \M; \boldsymbol\mu_0 = \boldsymbol\nu\}$, and defined by
$$
(S\boldsymbol\mu)^\tau_t = \nu^\tau \circ r^{-1}(t,\mu,\cdot),\quad \tau =1,2,
$$
where for any $\boldsymbol\mu \in \M_{\boldsymbol\nu}$ and any $x \in\R^d$, $r(t,\mu,x)$, with $\mu=\mu^1-\mu^2$,
is the unique solution of the following It\^{o} equation:
\begin{equation}\label{eq:sde-r}
\left\{
\begin{array}{lll}
dr(t)&=&  \int K(r(t),p)\mu_t(dp) dt + \int\Gamma(r(t),p)w(dp,dt),\\
r(0) &=& x.
\end{array}
\right.
\end{equation}
\end{defn}

Note that the measurability of the mapping $x \mapsto r(t,\mu,x)$
follows easily from the properties in Remark \ref{rem:hypotheses},
ensuring that $(S\boldsymbol\mu)^\tau_t$ is well-defined. In fact,
from the proof of Lemma \ref{lem:diffeo} below, the mapping $x
\mapsto r(t,\mu,x)$ is a homeomorphism, thus it is measurable.

In other words, for all $t\ge 0$ and for any bounded and measurable $f$ on $\R^d$,
$$
<(S\boldsymbol\mu)^\tau_t,f  > = \int f\{r(t,\mu,x)\}\nu^\tau (dx), \quad
\tau=1,2.
$$
In particular, $<(S\mu)_t,f  > = \int f\{r(t,\mu,x)\}\nu (dx)$.

The proof of existence and uniqueness of $r$ is similar to the proof of Lemma \ref{lem:sys-sde}, so it is omitted.\\


The following lemma is essential and confirms that the Hahn-Jordan
decomposition is preserved by the mapping $S$.

\begin{lemma}\label{lem:diffeo}
If $\boldsymbol\nu = \left(\nu^1,\nu^2\right)$ is the Hahn-Jordan decomposition
of $\nu$, then $(S\boldsymbol\mu)_t =
\left((S\boldsymbol\mu)_t^1,(S\boldsymbol\mu)_t^2\right)$ is the Hahn-Jordan decomposition
of $(S\boldsymbol\mu)_t^1-(S\boldsymbol\mu)_t^2$.
\end{lemma}

\begin{proof} The conditions in Remark \ref{rem:hypotheses} ensure that the ``local characteristics'' of the semimartingale
$B(x,t)+M(x,t)$, with $B(x,t) = \int_0^t K(x,p)\mu_s(dp)ds$ and $M(x,t) = \int_0^t \int \Gamma(x,p)w(dp,ds)$ satisfy the hypotheses of
\citet[Theorem 4.5.1]{Kunita:1990}. It follows that $x \mapsto r(t,\mu,x)$ is a
homeomorphism and therefore injective.
The results then follows from Proposition \ref{prop:bijection} in
Appendix \ref{app:aux}. \end{proof}

The next result, similar to Lemma  \ref{lem:extension}, is needed
to show that $S$ has a unique fixed point in $\M_{\boldsymbol\nu}$.

\begin{lemma}\label{lem:contraction}
For an any $T>0$, there exist $c'=c'(T)>0$, independent of $\boldsymbol \nu \in
M$, such that, if both $\boldsymbol\mu$ and $\boldsymbol\eta$ belong to $\M_{\boldsymbol\nu}$, then they satisfy
\begin{equation}\label{eq:contraction}
\gamma_{[0,T]}^2 (S\boldsymbol\mu, S\boldsymbol\eta) \leq c' \int_0^T E
\gamma_{2}^{2}(\boldsymbol\mu_t,\boldsymbol\eta_t)dt \le C' T \gamma_{[0,T]}^2
(\boldsymbol\mu,\boldsymbol\eta).
\end{equation}
\end{lemma}

The proof is given in Appendix \ref{app:cont}. 
%
%
%

One can now prove that $S$ has a unique fixed point.

\begin{thm}\label{thm:contraction}
Let $\boldsymbol\nu \in M$ be given. Then the mapping $S$ has a unique fixed
point $\boldsymbol\mu \in \M_{\boldsymbol\nu}$ given by $\boldsymbol\mu = \Psi(\boldsymbol\nu)$, and for any
$\boldsymbol\eta\in \M_{\boldsymbol\nu}$, $S^n \boldsymbol\eta$ converges to $\boldsymbol\mu$ as $n\to\infty$.
\end{thm}

\begin{proof} First, using It\^{o}'s formula, it is easy to check that
$\Psi(\boldsymbol\nu)$ is a fixed point of $S$.

It follows from Lemma \ref{lem:contraction} that $S$ is continuous
and that for any $\boldsymbol\mu\in \M_{\boldsymbol\nu}$, the sequence $\boldsymbol\mu_n = S^n \boldsymbol\mu$ is
Cauchy. Since the space $\M_{\boldsymbol\nu}$ is complete and the mapping $S$ is
continuous, $\boldsymbol\mu_n$ converges to $\boldsymbol\mu\in \M_{\boldsymbol\nu}$ which must be a
fixed point of $S$. Hence the set of fixed points of $S$ is not
empty. Next, it follows easily from Lemma \ref{lem:contraction} that
there are no more than one fixed point. For if $\boldsymbol\mu$ and $\boldsymbol\eta$ are
fixed points, then
$$
E\disp \sup_{0\leq t\leq T}\gamma_{2}^{2}(\boldsymbol\mu_t, \boldsymbol\eta_t)\leq c'
\int_0^T E \gamma_{2}^{2}(\boldsymbol\mu_t,\boldsymbol\eta_t)dt,
$$
so using Gronwall's inequality, one may conclude that $\boldsymbol\mu=\boldsymbol\eta$.
\end{proof}

Suppose that for any given $\boldsymbol\nu \in M$ and $\boldsymbol\mu \in \M_{\boldsymbol\nu}$, the
signed measure valued process $\chi = \chi^1-\chi^2$, with
$(\chi^1, \chi^2) \in\M_{\boldsymbol\nu}$, satisfies the weak (linear) evolution
equation
\begin{eqnarray}
<\chi_t,g> &=& <\nu,g> + \int_0^t  <\chi_s, L(\mu_s) g >ds \nonumber\\
&& \qquad +  \int_0^t \int <\chi_s, \nabla
g(\cdot)^\top\Gamma(\cdot ,p)> w(dp,ds),
\label{eq:signed-weak-linear}
\end{eqnarray}
for any nice $g$. \\

First note that such a solution exists. In fact, $\chi =
\chi^1-\chi^2$, with $(\chi^1, \chi^2) = S{\boldsymbol\mu}$, is a solution of
\eqref{eq:signed-weak-linear}, by a simple application of It\^o's
formula applied to $g(r(t))$ when $g$ is sufficiently smooth.\\

Note also that equation \eqref{eq:signed-weak-linear} still makes sense when $\chi$ is a trajectory in the space of Schwartz's tempered distributions, for any
test function $g$. Let $H_p$ denote the Sobolev space of order $p$, as defined in  \citet{Dawson/Vaillancourt:1995}.

The proof of the next lemma is done in  Appendix \ref{app:dual}. Before stating it, let $ L_{\mu_t}^* $ be the (Schwartz distributional) adjoint operator to $ L_{\mu_t} $ in the classical functional analytic sense, i.e., for $\chi\in \H_{-q}$, $\phi\in H_q$, $<\chi,L_{\mu_t}\phi> = <L_{\mu_t}^*\chi,\phi>$. See, e.g., \citet{Rudin:1973}.

\begin{lemma}\label{lem:dual} Suppose that for all $t\in [0,T]$ and any given $p\ge 1$, there exists $q\ge p$ and a positive constant $C_p$ so that
for every solution $\chi \in H_{-p}$ of \eqref{eq:signed-weak-linear},
\begin{equation}\label{eq:mono}
E\left\{<\chi,L_{\mu_t}^*\chi>_{-q} \right\} + \int_{\R^d} \left\| \nabla^* \Gamma(\cdot,x)\chi\right\|^2_{-q}dx \le C_p \|\chi\|^2_{-q}.
\end{equation}
Then \eqref{eq:signed-weak-linear} has at most one solution in trajectorial sense starting  from $\nu \in H_{-p}$.

\end{lemma}

As a consequence of the previous results, one obtains the uniqueness of the solution of the
stochastic evolution equation \eqref{eq:WSNSE} starting from
$\nu\in SM(m_1,m_2)$.

\begin{thm}\label{thm:uniqueness}
Suppose that $\nu\in SM(m_1,m_2)$ has Hahn-Jordan
decomposition $\boldsymbol\nu =(\nu^1,\nu^2)\in M$. Then, under the conditions stated in Remark \ref{rem:hypotheses} and Lemma \ref{lem:dual},
the stochastic evolution equation
\eqref{eq:WSNSE}, with initial condition $\nu$,  has a unique solution $\chi$ which preserves the mass, that is, for every $t\ge 0$,
there holds $\chi_t \in SM(m_1,m_2)$ and its Hahn-Jordan
decomposition is $\boldsymbol\chi_t$, where $\boldsymbol\chi = \Psi(\boldsymbol\nu)$.
\end{thm}

\begin{proof}

For every $\boldsymbol\nu \in M$, Theorem \ref{thm:extpsi} yields the existence of a solution $\mu = \mu^1-\mu^2$  to \eqref{eq:WSNSE}
starting from $\nu$, where $\boldsymbol\mu =( \mu^1,\mu^2) =\Psi(\boldsymbol\nu)$.
In order to show uniqueness, it suffices to prove that $ \bmu = S\bmu$ invariably ensues, that is,
any solution is a fixed point of $S$, which we already know has a unique fixed point by Theorem \ref{thm:contraction}.

First, it is easy to check that for any $g\in C_b^2(\R^d)$,
$(S\bmu)^1-(S\bmu)^2$ satisfies \eqref{eq:signed-weak-linear}, for any given $ \bmu $.

Now take $ \bmu $ to be any solution to \eqref{eq:WSNSE}.
To show that $(S\bmu)_t^1-(S\bmu)_t^2 = \mu_t$, notice that
$\eta_t = (S\bmu)_t^1-(S\bmu)_t^2 - \mu_t$  satisfies
\eqref{eq:signed-weak-linear} with $\bnu\equiv 0$. Applying Lemma  \ref{lem:dual}, one may conclude that $P\{ \eta_t = 0\}=1$ holds for every $t\ge 0$.
Hence $\mu_t = (S\bmu)_t^1-(S\bmu)_t^2$ for every $t\ge 0$ with probability one.

It follows from Lemma \ref{lem:diffeo} that $(S\bmu)_t $ is the
Hahn-Jordan decomposition of $(S\bmu)_t^1-(S\bmu)_t^2$. Hence, by
Proposition \ref{prop:jordan}, we have $\bmu_t = (S\bmu)_t$ since
$\mu_t = (S\bmu)_t^1-(S\bmu)_t^2$. Therefore, one may conclude that
the stochastic evolution equation has a unique solution $\mu$, with
Hahn-Jordan decomposition $\bmu=S\bmu$. This completes the proof.
\end{proof}


We can at this point discuss some of the more recent papers of
Kotelenez and his school, as requested by one of the referees, whom
we thank for drawing them to our attention. The main statement in
\citet[Theorem 3.3]{Kotelenez:2010} is similar to our Theorem
\ref{thm:uniqueness} except for the uniqueness. However the proof in
\citet[Theorem 3.3]{Kotelenez:2010} is incomplete since the
Hahn-Jordan decomposition (corresponding to our Lemma
\ref{lem:diffeo}) and stated in  \citet[Lemma 3.1]{Kotelenez:2010},
is proven only for discrete measures  in \citet[Corollary
2.6]{Kotelenez:2010}.

Next, in \citet{Kotelenez/Seadler:2011}, the main statement is
Theorem 3.5,   which is our Theorem \ref{thm:extpsi}, with the
additional claim that the Hahn-Jordan decomposition is preserved.
However the proof given there is incorrect since they do not prove
that the sequence of signed measures is convergent with respect to
the distance $\lambda$ defined by \eqref{eq:KSdistance}. It is an
example of the kind of error we mentioned in Remark
\ref{rem:KSdistance}.

Finally, in \citet{Kotelenez/Seadler:2012}, the main statement is
Theorem 3.5,  which is an extension of \citet[Theorem
3.3]{Kotelenez:2010}, with an additional claim about the
conservation of the mass of the Hahn-Jordan decomposition. The proof
of the latter is incomplete, since it is based on their Corollary
A.7, whose proof is missing, claimed to be a direct consequence of
\citet[Theorem A.6]{Kotelenez/Seadler:2012}. Furthermore, the proof
of the Hahn-Jordan decomposition in \citet[Theorem
3.5]{Kotelenez/Seadler:2012} is based on their Lemma 3.2, whose
proof is circular with their Corollary 3.3 upon which it implicitly
relies. More precisely, their Lemma 3.2 is a statement about the
dominance of the distance $\lambda$ for the mapping $S$ with respect
to the initial measure. To be proven, one absolutely needs a result
about the Hahn-Jordan decomposition of $S\boldsymbol\mu$, which is
stated as a Corollary of Lemma 3.2.

\section{Two-dimensional vorticity equations as special cases}\label{sec:relation}

The classical Navier-Stokes equations of fluid dynamics for the two dimensional velocity field $v(t,x)=(v_{1},v_{2})(t,x)\in \R^2$
of an incompressible viscid planar fluid submitted to a pressure field $p(t,x)\in \R^2$, with prescribed initial velocity $v_0$, are given by
\begin{equation} \label{eq:NS0}
\begin{array}{rll}
\frac{\partial }{\partial t}v+(v\cdot \nabla )v + \nabla p -\nu
\Delta v &=&0,
\qquad (t,x) \in  \R^1 \times \R^2, \\
& &\\
(\nabla\cdot v)(t,x)=\frac{\partial}{\partial x_{1}}v_1
+\frac{\partial}{\partial x_{2}}v_2 &=& 0, \qquad (t,x) \in  \R^1 \times \R^2,
\end{array}
\end{equation}
with initial and boundary conditions
\begin{equation} \label{eq:NS0cond}
\begin{array}{rll}
v(x,0) &=& v_0(x) ,  \qquad x \in \R^2, \\
& &\\
\lim_{|x|\to\infty } v(t,x) &=& 0, \qquad t \in \R^1.
\end{array}
\end{equation}

Here the constant $\nu\geq0$ denotes the kinematic viscosity coefficient and we write $x=(x_{1},x_{2})\in {\mathbb R}^2$.
$\nabla=(\frac{\partial }{\partial x_{1}},\frac{\partial}{\partial x_{2}})$ is the gradient and
$\Delta=\nabla\cdot\nabla$ is the Laplace operator.

A great deal of information about the solution $v$ can be gleaned from a
scalar parameter called the vorticity (or rotation) $\omega $ of the two dimensional flow, defined as
$$
\omega := {\rm rot} \; v = \frac{\partial }{\partial
x_{1}}v_{2}-\frac{\partial }{\partial x_{2}}v_{1}.
$$

Using \eqref{eq:NS0cond} and treating \eqref{eq:NS0} formally, one obtains the
corresponding (pressure invariant) two-dimensional vorticity equations
\begin{equation} \label{eq:vnse2d}
\begin{array}{rll}
\frac{\partial}{\partial t}\omega+(v\cdot\nabla)\omega-\nu
\Delta\omega &=&0,
\qquad (t,x) \in  \R^1 \times \R^2, \\
& &\\
(\nabla\cdot v)(t,x) &=& 0, \qquad (t,x) \in  \R^1 \times \R^2.
\end{array}
\end{equation}

Introducing the operator $\nabla^{\bot}=(-\frac{\partial}{\partial
x_{2}},\frac{\partial }{\partial x_{1}})$ yields, by virtue of $\nabla\cdot
v=0$, the classical two dimensional Biot-Savard formula
\begin{equation}\label{2drep}
v(t,x)=\int(\nabla^{\bot}g)(x-y)\omega(t,y)dy,
\end{equation}
with $g(r)=h(\|r\|)$, with $h(s) = \frac{1}{2\pi}\ln{s} $, $s>0$.

Making sense of \eqref{eq:NS0}, \eqref{eq:vnse2d} and even
\eqref{2drep} requires the precise identification of the spaces of
values wherein lie $v$ and $\omega$. There is an extensive
literature on the conditions of existence and uniqueness of solution
for both equations \eqref{eq:NS0} and \eqref{eq:vnse2d}, when the initial data is restricted to either a
nice enough function space or a small subset of the space of all
Borel signed measures. See \citet{Chorin/Marsden:1993},
\citet{Ben-Artzi:2003} and \citet{Gallagher/Gallay:2005} for precise
statements as well as some historical background. The existence of
solutions for the rougher initial data selected arbitrarily within
the space of signed measures, a legitimate requirement for the sake
of good statistical modeling of particle behaviour at the
microscopic level, is harder and was first accomplished rigorously
for \eqref{eq:vnse2d} by \citet{Cottet:1986}. Uniqueness of solution
for \eqref{eq:vnse2d} took even longer to achieve and is due to
\citet{Ben-Artzi:2003} and \citet{Gallagher/Gallay:2005}. In both
cases, the solution is only shown to have continuous trajectories
away from 0, a consequence of both the singularity at 0 of $g$ and
the use of the total variation norm to induce a manageable topology
on the state space. To our knowledge mass conservation has never
been proved rigorously for these equations.

\begin{rem}
Conditions in \eqref{eq:gamma} cover the vorticity equations \eqref{eq:vnse2d}
only after the singularity at 0 of kernel $g$
is removed through smoothing, which we explain next, using the same notation as in \cite{Marchioro/Pulvirenti:1982}.
Let $g_{\varepsilon}(r) = h_{\varepsilon}(\|r\|)$, $0<\varepsilon
\leq 1$ where
 $h_{\varepsilon} \in C_{b}^{2}({\mathbb R})$ is selected so as to
satisfy $h_{\varepsilon}(s)=h(s)$ for $\varepsilon \leq s
\leq\frac{1}{\varepsilon}$ ,
$h_{\varepsilon}'(0)=0$, and for all $s>0$, $|h_{\varepsilon}'(s)|
\leq |h'(s)|$, $| h_{\varepsilon}''(s)| \leq |h''(s)|$.
Such a filter of smooth approximations is easy to build,
e.g.  see \citet{Leonard:1985}. For $r\neq 0$, set
$K_{\varepsilon}(r)=(\nabla^{\bot}g_{\varepsilon})(r)$$, r\in
{\mathbb R}^2$. It follows from the assumption
$h_{\varepsilon}'(0)=0$ that $K_{\varepsilon}(0)=0$ makes
$K_{\varepsilon}$ continuous on $\mathbb{R}^2$. Moreover, since
$h_{\varepsilon}'(0)=0$ and $h_{\varepsilon}''$ is bounded by
$C_\varepsilon$ (say), it follows that
$|K_{\varepsilon}(r)-K_{\varepsilon}(q)| \le 2C_\varepsilon |r-q|$,
that is $K_{\varepsilon}$ is Lipschitz.  Finally, note that the monotonicity condition \eqref{eq:mono}
holds for these equations by \citet{Mikulevicius/Rozovskii:2005}[Proposition 2.12],
provided that in addition, $\Gamma$ and $K_\varepsilon$ are bounded.
\end{rem}

The regularized or smoothed vorticity equations are then given by
\eqref{eq:vnse2d} with $v$ no longer being a solution to
\eqref{eq:NS0} but rather of the form \eqref{2drep} with
$\nabla^{\bot}g$ replaced by the approximating
$\nabla^{\bot}g_{\varepsilon}$. Since $\nabla \cdot K_\varepsilon
\equiv 0$, their weak form may be written as
\begin{equation}\label{eq:weak-smooth-NS}
\left\{
\begin{array}{rll}
d<\!\omega_t,f\!> &=& <\!\omega_t, L_\varepsilon(\omega_t) f \!>dt \\
\\
U_\varepsilon (r,\omega_t) &=& \int K_\varepsilon(r-q)\omega_t(dq),\\
\end{array}
\right.
\end{equation}
where, for any $x\in\mathbb{R}^2$,
$$
L_\varepsilon (\omega_t)f(x) = \sum_{j=1}^{2} \partial_{x_j}f(x)(U_\varepsilon)_j (x,\omega_t) +\nu \Delta f(x),
$$
with $<\!\omega,f\!>$ is just the integral of $f$ with respect to the measure (or density) $\omega$.
The advantage of \eqref{eq:weak-smooth-NS} over \eqref{eq:NS0} or \eqref{eq:vnse2d} is that it makes perfect sense
for finite signed measures $\omega_t$ and offers a ready-made stochastic version under the guise of our \eqref{eq:WSNSE}.


This reformulation allows us to provide through Theorem
\ref{thm:uniqueness} the first rigorous statements and proofs in the
matters of existence, uniqueness, mass conservation and continuity
at the origin for arbitrary initial data for both the regularized
vorticity equations \eqref{eq:weak-smooth-NS} and their stochastic
counterparts analyzed in \cite{Marchioro/Pulvirenti:1982} and
\cite{Kotelenez:1995}, among others. It must be pointed out that not
only did Marchioro and Pulvirenti claim to give explicit conditions
for equations \eqref{eq:weak-smooth-NS} to possess one and only one
solution, they did so even when it was perturbed by independent
Brownian motions. They also claimed to have proved mass conservation
(see the statement of their Theorem 2.1) but, just as in their
proofs of the two previous statements,  the use of an incomplete
state space mars their argument and the same difficulty affects the
proof of their Theorem 3.1 in the stochastic case. Specifically,
\cite{Marchioro/Pulvirenti:1982} defined a mapping $\mathcal{S}$ from
$SM(m_1,m_2)$ into the space of continuous functions
$C([0,T);SM(m_1,m_2))$ by setting
$$
(\mathcal{S}\mu)_t (A) =  \int P_t^{\mu}(A|y)\nu(dy)
$$
for all Borel sets $A\subset \R^2$, where $P_t^{\mu}(\cdot|y)$ are the
transition probabilities of the diffusion process, solution of the
stochastic differential equation
$$
d x_t = U_\epsilon(x_t,\mu_t)dt + \sigma dW_t,\qquad x_0=y,
$$
which is a particular case of equation \eqref{eq:sde-r}.
They claimed that it is obvious that $(\mathcal{S}\mu)_t\in
SM(m_1,m_2)$ if $\nu \in SM(m_1,
m_2)$; this is shown to be false in Appendix \ref{app:disprove}.
Next, the proof of their Theorem 2.1 relies on a fixed point theorem
for operators on complete spaces in an essential way. Unfortunately, their
statements are widely quoted and used in the literature.

In a similar fashion, \cite{Kotelenez:1995} claimed to be able to
extend Marchioro and Pulvirenti's treatment when the stochastic
terms involve Brownian sheets as the driving random environment, a
more realistic description from the physical standpoint since the
energy conservation law of physics is then respected at the
microscopic level. However several of his arguments involve
repeated use of the completeness of spaces that are simply not
complete. A further claim by Kotelenez to the effect that the
conservation of total positive and negative vortices follows from
his construction turns out to hinge on the (incorrect) completeness
assumption already mentioned. For his stochastic version of the weak
regularized vorticity equations \eqref{eq:weak-smooth-NS}, Kotelenez
selected the kernel $\Gamma$ appearing in our equations
\eqref{eq:WSNSE} and \eqref{eq:sys-part} amongst a family of simple
gaussian kernels with values inside the set of diagonal matrices.
They are easily shown to satisfy our conditions \eqref{eq:gamma}.

For the sake of both papers and the ensuing literature, we provide in the appendix
corrected statements and proofs of some of their claims.

We must finally draw the reader's attention to a clever
transformation introduced by \cite{Jourdain:2000} which enabled him
to transfer the problem of manipulating sequences of signed measures
to that of associated probability measures, thus obtaining existence
and uniqueness of solutions to some viscous scalar conservation laws
by a direct application of the propagation of chaos results of
\cite{Sznitman:1991}. This idea allowed \cite{Meleard:2001} to
provide a proof of uniqueness for the vorticity equation
\eqref{eq:vnse2d} under some fairly unrestricted conditions on the
initial measure.

Whether our methods extend to the original equation \eqref{eq:vnse2d} with the explosive kernel $g$ above remains an open question.

\appendix
\section{Auxiliary results}\label{app:aux}

This appendix contains the proofs of the more technical results of
this paper. By organizing the material of this paper in such a
fashion, our hope is that the reader will get a better overview of
the subject at hand without getting bogged down in small details
that could hamper his understanding of the connexions between the
various results and contributions.

%
%
%
%
%

%
%
%
%
%
%
%

\begin{prop}\label{prop:bijection}
If $T$ is a measurable injection on $\R^d$, then $\bmu\mapsto
\bmu\circ T$ preserves singularity, i.e., if $\bmu = (\mu^1,\mu^2)$ is
the Hahn-Jordan decomposition of some given signed measure
$\mu$, then $\bmu\circ T = (\mu^1\circ T,\mu^2\circ T)$ is the
Hahn-Jordan decomposition of $\mu\circ T$.
\end{prop}

\begin{proof} Let $A^1, A^2$ be disjoint Borel sets such that $\mu^1(A^1) =
m_1$, $\mu^1(A^2) = 0$,  $\mu^2(A^1) = 0$ and $\mu^2(A^2) = m_2$.
Set $B^\tau = T(A^\tau)$, $\tau=1,2$. Then, for $\tau=1,2$,
$$
\mu^1\circ T(B^\tau) = \mu^1\left\{T^{-1}\left(B^\tau\right)\right\}
= \mu^1\left(A^\tau\right)
$$
and
$$
\mu^2\circ T(B^\tau) = \mu^2\left\{T^{-1}\left(B^\tau\right)\right\}
= \mu^2\left(A^\tau\right).
$$
Next, remark that $B^1\cap B^2 = \emptyset$ since $T$ is an
injection. Hence the result.

\end{proof}

\begin{prop}\label{prop:jordan}
 Let $\bmu=(\mu^1,\mu^2), \bnu = (\nu^1,\nu^2)\in M$ be such that $\mu = \mu^1-\mu^2$ equals  $\nu=\nu^1 - \nu^2$.
 If $\bnu$ is the Hahn-Jordan decomposition of $\nu$,
then  $\bmu = \bnu$.
\end{prop}

 \begin{proof} Let $P_1$ and $P_2$ be the positive and negative
sets of $\nu$. Then $\nu^1(P_1)=m_1$, $\nu^2(P_2)=m_2$ and
$\nu^1 (P_2)=0 = \nu^2(P_1)$. It follows that
$$
\mu^1(P_1)-\mu^2(P_1) = \mu(P_1) = \nu(P_1) =
\nu^1(P_1)-\nu^2(P_1) = m_1
$$
and
$$
\mu^1(P_2)-\mu^2(P_2) = \nu^1(P_2)-\nu^2(P_2) = -m_2.
$$
Since $\mu^\tau (\R^d) = m_\tau \ge \mu^\tau(P_\tau)$ for
$\tau=1,2$, it follows that $\mu^1(P_1)=m_1$, $\mu^1(P_2)=0$,
$\mu^2(P_1)=0$ and $\mu^2(P_2)=m_2$. Therefore $\bmu = (\mu^1,\mu^2)$
is also the Hahn-Jordan decomposition of $\mu=\nu$.
Because the uniqueness of the Hahn-Jordan decomposition, one may conclude that $\bmu = \bnu$. \end{proof}

Recall that the minimal Lipschitz constant  for a function $f:\R^d\mapsto
\mathbb{R}$ is defined by
\begin{equation}\label{eq:lipschitz}
{\|f\|}_{L} = \sup_{r\neq q \in
\mathbb{R}^d}\frac{|f(r)-f(q)|}{\rho(r,q)}.
\end{equation}
It follows from the Kantorovich-Rubinstein Theorem
\citep[Theorem 11.8.2]{Dudley:1989} that for any $\mu,\nu\in
M(m)$,
$$
W_1(\mu,\nu) = m\sup_{{\|f\|}_L \le 1} \left|<\mu-\nu,f> \right|.
$$

If $\bmu=(\mu^1,\mu^2), \bnu =(\nu^1,\nu^2)\in M(m_1, m_2)$ then, by denoting $\mu=
\mu^1-\mu^2$, $\nu=\nu^1-\nu^2$ and $m= \max(m_1,m_2)$, note that
\begin{equation}\label{eq:kr1}
\sup_{{\|f\|}_L \le 1} \left|<\mu-\nu,f> \right| \le
\gamma_1(\bmu,\bnu)
\end{equation}
and
\begin{equation}\label{eq:kr2}
\sup_{{\|f\|}_L \le 1} <\mu-\nu ,f>^2 \; \le 2m^2
\gamma_2^2 (\bmu,\bnu).
\end{equation}


\section{Proof of the main results}

\begin{proof}[Proof of Lemma \ref{lem:sys-sde}]\label{app:kol}

 For any adapted and $P\otimes \lambda$ measurable
stochastic process $q=\left(q^{1},\ldots ,q^N \right)^\top \in
C([0,T]; \R^{dN})$, set
$$
\ Q_{q}(t)=\sum_{i=1}^{N}a_{i}\delta_{q^{i}(t)}
$$
and define the mapping $q \mapsto F(q) = \hat q$, where, for all
$i=1,\ldots, d$,  and every $t\ge 0$,
$$
\hat q^{i}(t)= r^{i}(0)+\int_{0}^{t}\int
K(q^{i}(s),p)Q_{q}(dp,s)ds +\int_{0}^{t}\int
\Gamma(q^{i}(s),p)w(dp,ds).
$$
One wants to show the existence and uniqueness by Picard's
iteration, using the same technique as in
\citet{Ethier/Kurtz:1986}.

To this end, for any adapted  $C([0,T]; \R^{dN})$-valued random
variables $q_1,q_2$, one needs to estimate
$\left\|\widehat{q}_{1}^{i}(t)-\widehat{q}_{2}^{i}(t)\right\|$ for
all $i=1,\ldots, N$. First, set
$$
A_l^i(t) =  \int_{0}^{t}\int K(q_{l}^{i}(s),p)Q_{q_l}(dp,s)ds =
\sum_{j=1}^{N} a_{j}\int_0^t K(q_{l}^{i}(s),q_{l}^{j}(s))ds,
$$
$l=1,2$. Then, for any $i=1,\ldots, N$,
\begin{multline*}
\|A_1^i(t)-A_2^i(t)\|^2 \\
=\left\|
\int_{0}^{t}\sum_{j=1}^{N}a_{j}\left\{K(q_{1}^{i}(s),q_{1}^{j}(s))-K(q_{2}^{i}(s),q_{2}^{j}(s)\right\}ds\right\|
^{2}\\
\leq T \int_{0}^{t}\left\| \sum_{j=1}^{N}a_{j}\left\{ K(q_{1}^{i}(s),q_{1}^{j}(s))-K(q_{2}^{i}(s),q_{2}^{j}(s)\right\}\right\| ^{2}ds \\
\leq NT \sum_{j=1}^{N} a_{j}^{2}\int_{0}^{t} \left\| K(q_{1}^{i}(s),q_{1}^{j}(s))-K(q_{2}^{i}(s),q_{2}^{j}(s))\right\|^{2}ds\\
\leq NT \|K\|_L^2 \sum_{j=1}^{N}a_{j}^{2}  \int_{0}^{t}\left\{\left\|(q_{1}^{i}(s)-q_{2}^{i}(s)\right\|^2 +\left\|q_{1}^{j}(s)-q_{2}^{j}(s)\right\|^{2}\right\}ds\\
\leq 2NT\|K\|_L^2 \left(\sum_{j=1}^{N}a_{j}^{2}\right)
\int_{0}^{t}\left\|  q_{1}(s)-q_{2}(s)\right\|  _{N}^{2}ds,
\end{multline*}
where ${\|r-q\|}_N = \disp \max_{1\leq i\leq N}\|r^{i}-q^{i}\|$.

Next, using Doob's inequality,
\begin{multline*}
E\disp \sup_{0 \le u \leq t}\left\|  \int_{0}^{u}\int
\Gamma(q_{1}^{i}(s),p)w(dp,ds)-\int_{0}^{u}
\int\Gamma(q_{2}^{i}(s),p)w(dp,ds)\right\|^{2}\\
\le 4  \sum_{j=1}^d E\disp \left< \left<
\int_{0}^{t}\int \sum_{l=1}^{d} \left\{\Gamma_{jl}(q_{1}^{i}(s),p)-\Gamma_{jl}(q_{2}^{i}(s),p)\right\}w_l(dp,ds)\right> \right> \\
= 4\sum_{j=1}^{d} \sum_{l=1}^{d} E\disp \int_{0}^{t}\int \left\{ \Gamma_{jl}(q_{1}^{i}(s),p)-\Gamma_{jl}(q_{2}^{i}(s),p)\right\}^{2}dpds \\
\leq  4C_\Gamma^2     \int_{0}^{t} E
\|q_{1}^{i}(s)-q_{2}^{i}(s)\|^2 ds ,
\end{multline*}
where \eqref{eq:gamma} was used in the last chain of inequalities.
Hence, setting $C = 8C_\Gamma^2 + 4NT\|K\|_L^2
\left(\sum_{j=1}^{N}a_{j}^{2}\right)$, one obtains
\begin{equation}\label{eq:rec}
E \disp \sup_{0 \le s \leq t} \left\|\hat q_1(s)-\hat q_2(s)
\right\|_N^2 \le C\int_0^t E \|q_1(s)-q_2(s)\|_N^2 ds.
\end{equation}

Therefore, if $X_0(t) \equiv r(0)$ and $X_{k+1} = F(X_k)$, one
obtains, for all $t\in [0,T]$,
\begin{eqnarray*}
E \disp \sup_{0 \le s\leq t} \|X_{1}(s)-X_0(s)\|_N^2 &\le&   2t^2\max_{1\le i\le N} \left\| \int K(r^i(0),p)Q_0(dp)\right\|^2\\
&& \qquad +
8t\max_{1\le i\le N} \sum_{j=1}^d\sum_{l=1}^d \int \Gamma_{jl}^2(r^i(0),p)dp\\
& \le & C_1 t,
\end{eqnarray*}
for some constant $C_1$. Next, using the last  equality together
with \eqref{eq:rec}, one obtains
$$
E  \sup_{s\in [0,t]} \|X_{k+1}(s)-X_k(s)\|_N^2 \le \frac{C_1}{C}
\frac{(Ct)^{k+1}}{(k+1)!}, \quad k\ge 0.
$$
It follows from Borel-Cantelli Theorem that with probability one,
$$
\|X_{k+1}(s)-X_k(s)\|_N^2 \le 2^{-(k+1)}
$$
for all large $k$. Since $C([0,t];\R^{dN})$ is complete under the
sup norm, it follows that $X_k$ converges
 almost surely to $X \in C([0,t];\R^{dN})$. Since $F$ is a continuous mapping, one has $F(X)=X$ so $X$ is a solution.
 Finally, if $X$ and $Y$ are two solutions, i.e. $F(X)=X$ and $F(Y)=Y$, \eqref{eq:rec} yields
 $$
E  \sup_{s\in [0,t]} \|X(s)-Y(s)\|_N^2 \le C\int_0^t E  \sup_{u\in
[0,s]} \|X(u)-Y(u)\|_N^2du
 $$
so Gronwall's inequality entails that $E  \sup_{s\in [0,t]}
\|X(s)-Y(s)\|_N^2=0$, proving uniqueness. \end{proof}

\begin{proof}[Proof of Lemma \ref{lem:extension}]\label{app:extension}

Consider the following two $\R^d$-valued It\^{o} equations with
deterministic initial conditions $x_0,y_0\in \R^d$.
$$
\left\{
\begin{array}{lll}
dr(t)&=&\int K(r(t),p)\chi_t (dp)dt+\int\Gamma(r(t),p)w(dp,dt),\\
r(0) &=& x_0;\\
dq(t)&=& \int K(q(t),p)\eta_t(dp)dt+\int \Gamma (q(t),p)w(dp,dt),\\
q(0) &=& y_0.
\end{array}
\right.
$$

When $x_0 =x^{i}(0)$, then $r(t,x_0)  = x^{i}(t)$, using
uniqueness property in Lemma \ref{lem:sys-sde}. Similarly,
$q(t,y_0) = y^{i}(t)$, if $y_0 = y^i(0)$.

Let $Q^\tau_0$ be  joint representations of $\left(\chi^\tau_0,
\eta^\tau_0 \right)$, $\tau=1,2$. The following expressions define
joint representations $Q^\tau_t $ of $\left(\chi^\tau_t,
\eta^\tau_t\right)$ for every $t\ge 0$ and $\tau=1,2$: for $f\in C_{b}({\mathbb R}^{2d})$ set
$$
\int\int f(y,z)Q^\tau_t(dy,dz)=\int\int
f(r(t,y),q(t,z))Q^\tau_0(dy,dz).
$$

To see that $Q^1_t$ is indeed a representation for $\left(\chi^1_t,
\eta^1_t\right)$, remark that
\begin{eqnarray*}
\int\int f(y)Q^1_t(dy,dz)&= &\int\int f(r(t,y))Q^1_0(dy,dz)\\
&= & m_2 \int f(r(t,y))\chi^1_0(dy)\\
&=& m_2 \sum_{i; a_i\ge 0}a_i f(r(t,x^i(0)))\\
&=& m_2\sum_{i; a_i\ge 0} a_i f(x^{i}(t))\\
&= & m_2\int f(y)\chi^1_t(dy).
\end{eqnarray*}

Similarly,
$$
\int f(y)Q^2_t(dy,dz) = m_2\int f(y)\chi^2_t(dy)
$$
and
$$
\int f(z)Q^\tau_t (dy,dz) = m_1\int f(z)\eta^\tau_t (dz), \quad \tau \in \{1,2\}.
$$
It follows that if $\boldsymbol\chi = (\chi^1,\chi^2)$ and $\boldsymbol\eta = (\eta^1,\eta^2)$, then
\begin{eqnarray*}
\gamma_2^2(\boldsymbol\chi_t,\boldsymbol\eta_t) &\le&  \int \rho^2(r(t,y),q(t,z))Q^1_0(dy,dz) \\
&& \quad + \int \rho^2(r(t,y),q(t,z))Q^2_0(dy,dz).
\end{eqnarray*}

Also,
\begin{eqnarray*}
\|r(t)-r(0) -q(t)+q(0)\| &\ge & \rho(r(t)-q(t),r(0)-q(0))\\
&\ge & \rho(r(t),q(t))-\rho(r(0),q(0)).
\end{eqnarray*}

Next we calculate $\left\|  r(t)-r(0)-q(t)+q(0)\right\|^{2}$. First, set $\chi = \chi^1-\chi^2$ and $\eta = \eta^1-\eta^2$.

Proceeding as in Lemma \ref{lem:sys-sde}, one has
\begin{eqnarray*}
E  \sup_{0\leq t\leq T}\left\|  \int_{0}^{t}\int
\Gamma(r(s),p)w(dp,ds)-\int_{0}^{t}\int
\Gamma(q(s),p)w(dp,ds)\right\|  ^{2} &&\\
  \qquad \le 2c   \int_{0}^{T}E\rho^{2}(r(s),q(s))ds &&
\end{eqnarray*}
and
\begin{eqnarray*}
\left\|  \disp \int_{0}^{t}\int K(r(s),p)\chi_s(dp)ds
-\int_{0}^{t}\int K(q(s),p)\eta_s(dp)ds\right\|^{2}&& \\
\qquad \leq  2\left\| \disp  \int_{0}^{t}\int(K(r(s),p)-K(q(s),p))\chi_s(dp)ds\right\|  ^{2} &&  \\
+2\left\| \int_{0}^{t}\int K(q(s),p)(\chi_s
(dp)-\eta_s(dp))ds\right\|  ^{2}. &&
\end{eqnarray*}

Recall the notation $m= \max(m_1,m_2)$. Since ${\|K\|}_L \le C$, one has that
\begin{eqnarray*}
\left\|  \disp \int_{0}^{t}\int(K(r(s),p)-K(q(s),p))\chi_s(dp)ds\right\|  ^{2} && \\
\qquad \leq 4T \disp (mC)^{2} \int_{0}^{t}\rho^{2}(r(s),q(s))ds.
&&
\end{eqnarray*}

In addition, it follows  from \eqref{eq:kr2} that
\begin{eqnarray*}
\left\| \int_{0}^{t}\int K(q(s),p)(\chi_s(dp)-\eta_s(dp))ds\right\|^2  &&\\
\qquad \le 2T (mC)^{2} \disp \int_0^t \gamma_2^2(\boldsymbol\chi_s,\boldsymbol\eta_s)ds. && \\
\end{eqnarray*}
Hence
\begin{eqnarray*}
E \sup_{0\leq t\leq T}\rho^{2}(r(t),q(t))& \leq &  2E\disp
\sup_{0\leq t\leq T}\left\| r(t)-r(0)-q(t)+q(0)\right\|
^{2}+2\rho^{2}(y,z) \\ & \leq & 2\rho^{2}(y,z) \\ && \quad + C_1
\int_{0}^{T}E\disp \sup_{0\leq s\leq T}\rho ^{2}(r(s),q(s))ds \\
&& \qquad + C_1 \int_{0}^{T}\gamma_{2} ^{2}(\boldsymbol\chi_s,\boldsymbol\eta_s)ds,
\end{eqnarray*}
where $C_1 = 16T \left(c+(2mC)^2\right)$.

By Gronwall's inequality, we have
$$
E\disp \sup_{0\leq t\leq T}\rho^{2}(r(t),q(t))  \leq 2
e^{C_1T}\rho^{2}(y,z)+ e^{C_1 T}
\int_{0}^{T}\gamma_{2}^{2}(\boldsymbol\chi_s,\boldsymbol\eta_s)ds.
$$

Integrating the last inequality with respect to $Q^1_0 +Q^2_0$, one
obtains
\begin{eqnarray*}
E  \sup_{0\leq s\leq t}\gamma_2^{2}(\boldsymbol\chi_t,\boldsymbol\eta_t) & \leq & E\disp
\sup_{0\leq t\leq T}\rho^{2}(r(t),q(t)) \{Q^1_0 +Q^2_0\}(dy,dz)
\\ & \leq & 2 e^{C_1 T} \int \rho^{2}(y,z) \{Q^1_0 +Q^2_0\}(dy,dz)  \\ &&
\quad + m e^{C_1 T}  \int_{0}^{T}\gamma_{2}^{2}(\boldsymbol\chi_s,\boldsymbol\eta_s)ds.
\end{eqnarray*}
Taking the infimum over all $Q^\tau_0 \in
\H(\chi^\tau_0,\eta^\tau_0)$, $\tau=1,2$, one gets
\begin{eqnarray*}
E  \sup_{0\leq s\leq t}\gamma_2^{2}(\boldsymbol\chi_t,\boldsymbol\eta_t) &\le & 2e^{C_1
T} \gamma_2^{2}(\boldsymbol\chi_0,\boldsymbol\eta_0) \\ && \quad + m e^{C_1 T}
\int_{0}^{T}\gamma_{2}^{2}(\boldsymbol\chi_s,\boldsymbol\eta_s)ds.
\end{eqnarray*}

Using  Gronwall's inequality again yields \eqref{eq:relation}.
Finally, \eqref{eq:kr3} is obtained by combining \eqref{eq:kr2} and
\eqref{eq:relation}. \end{proof}

\begin{proof}
[Proof of Theorem \ref{thm:extpsi}]\label{app:pfthmpsi}

To this end, let  $\boldsymbol\chi_0 \in \mathcal{M}$ be given. Since
$\mathcal{M}^{(f)}$ is dense in $\mathcal{M}$, there exists a
sequence $\boldsymbol\chi_{0,n}\in \mathcal{M}^{(f)}$ so that
$\gamma(\boldsymbol\chi_0,\boldsymbol\chi_{0,n})\to 0$, as $n\to \infty$. Using Lemma
\ref{lem:extension}, it follows that $\boldsymbol\chi_n = \Psi(\boldsymbol\chi_{0,n})$ is
a Cauchy sequence in $\M$. Thus there exists $\boldsymbol\chi \in \M$ so that
$\gamma_{[0,T]}(\boldsymbol\chi,\boldsymbol\chi_n) \to 0$. Since the limit does not depend
on the sequence, the mapping is well-defined.

It also follows that for any $\boldsymbol\chi_0,\boldsymbol\eta_0 \in \mathcal{M}$, the
corresponding paths $\boldsymbol\chi,\boldsymbol\eta \in \mathcal{M}_{[0,T]}$ satisfy
$$
\gamma_{[0,T]}(\boldsymbol\chi,\boldsymbol\eta) \le  c' \gamma(\boldsymbol\chi_0,\boldsymbol\eta_0).
$$

Next one will show that this extension gives a weak solution of
the stochastic evolution equation \eqref{eq:WSNSE}.

Using \eqref{eq:kr3} and the last inequality, one can conclude
that for any $f$ such that $\|f\|_L \le 1$,  one has
$$
E\sup_{0\le t\le T} \sup_{{\|f\|}_L \le 1}
<\chi_t-\eta_t,f>^2 \; \le c''
\gamma(\boldsymbol\chi_0,\boldsymbol\eta_0),
$$
where $\chi = \chi^1-\chi^2$ and $\eta = \eta^1-\eta^2$.
In particular,
$$
\lim_{n\to\infty} E\sup_{0\le t\le T} \sup_{{\|f\|}_L \le 1}
<\chi_t-\chi_{t,n} ,f>^2  = 0.
$$

The next step is to show that for any $f\in C_{b}^{2}({\mathbb
R}^d)$ and $\boldsymbol\chi_0\in M$, $ \chi$ satisfies
\eqref{eq:WSNSE}.

Note that by the choice of $f$, both $f$ and $\nabla f$ are
bounded, so the right-hand side of the stochastic evolution
equation  is defined for $<\chi_t,f>$. Moreover, since
$$
{\left\|f\right\|}_L <\infty \quad \mbox{and} \quad {\left\|
\nabla f\right\|}_L  <\infty,
$$
it follows that
$$
\lim_{n\to\infty} E\sup_{0\le t\le T}
<\chi_t-\chi_{t,n},f>^2  = 0.
$$

Similarly,
$$
\disp E\sup_{0\le t\le T} \sup_{p}
|U(p,\chi_t)-
U(p,\chi_{t,n})|^2\to 0
$$
and  one can prove that all righthand side terms of \eqref{eq:WSNSE}
tends to zero, as $n$ tends to infinity. This completes the proof.
\end{proof}

\begin{proof}[Proof of Lemma \ref{lem:contraction}]\label{app:cont}

For any $\bmu \in \M_{\boldsymbol\nu}$, $\mu = \mu^1-\mu^2$, and any $x \in\R^d$, let $r(t,\mu,x)$ be
the unique solution of the following It\^{o} equation:
$$
\left\{
\begin{array}{lll}
dr(t)&=&  \int K(r(t),p)\mu_t(dp) dt + \int\Gamma(r(t),p)w(dp,dt),\\
r(0) &=& x.
\end{array}
\right.
$$
The proof of existence and uniqueness is similar to that of Lemma \ref{lem:sys-sde}.\\

Next, recall that the operator $S$, acting on $\bmu \in \M_{\boldsymbol\nu}$, is
defined by
$$
(S\bmu)^\tau_t = \nu^\tau \circ r(t,\mu,\cdot),\quad \tau=1,2.
$$

For $\tau=1,2$, let $Q^\tau$ be  joint representations for
$\left(\nu^\tau, \nu^\tau\right)$ and define, for any $\bmu, \boldsymbol\eta\in
\M_{\boldsymbol\nu}$, the joint representations $Q^\tau_t $ for $ (S\bmu)^\tau_t$
and $ (S\boldsymbol\eta)^\tau_t$ as follows: for every
$f\in C_{b}({\mathbb R}^{2d})$
$$
\int\int f(y,z)Q^\tau_t(dy,dz)=\int\int f(r(t,\mu,y),r(t,\eta,
z))Q^\tau(dy,dz).
$$

To see that $Q^\tau_t$ is indeed a representation for $\left(
(S\bmu)^\tau_t, (S\boldsymbol\eta)^\tau_t\right)$, $\tau=1,2$, remark that
\begin{eqnarray*}
\int\int f(y)Q^\tau_t(dy,dz)&= &\int\int f(r(t,\mu,y))Q^\tau(dy,dz)\\
&= & m_2\int f(r(t,\mu,y))\nu^\tau(dy)\\
&=& m_2<(S\bmu)_t^\tau,f>.
\end{eqnarray*}

Similarly, for $\tau=1,2$,
\begin{eqnarray*}
\int\int f(z)Q^\tau_t(dy,dz)&= &\int\int f(r(t,\eta,z))Q^\tau(dy,dz)\\
&= & m_1\int f(r(t,\eta,z))\nu^\tau(dz)\\
&=& m_1<(S\boldsymbol\eta)_t^\tau,f>.
\end{eqnarray*}

It follows that
\begin{eqnarray*}
\gamma_2^2((S\bmu)_t,(S\boldsymbol\eta)_t) &\le&  \int \rho^2(r(t,\mu,y),r(t,\eta,z))Q^1(dy,dz) \\
&& \quad + \int \rho^2(r(t,\mu,y),r(t,\eta,z))Q^2(dy,dz).
\end{eqnarray*}

The rest of the proof is almost identical to the proof of Lemma \ref{lem:extension} so it is omitted.
\end{proof}

\begin{proof}[Proof of Lemma \ref{lem:dual}]\label{app:dual}
The proof is classical. Suppose there are two solutions, and write $\chi $ for their difference.
It is easy to check that for any $\phi\in H_q$,
\begin{eqnarray*}
<\chi_t,\phi> &=& \int_0^t  <\chi_s, L(\mu_s) \phi >ds \\
&& \qquad +  \int_0^t \int <\chi_s, \nabla
\phi(\cdot)^\top\Gamma(\cdot ,p)> w(dp,ds).
\end{eqnarray*}
Therefore applying Ito's formula to $e^{-t C_p}<\chi_t,\phi>^2$ and taking expectations, one ends up with
\begin{eqnarray*}
e^{-2t C_p} E \left\{<\chi_t,\phi>^2 \right\}  &= &
2   E \left\{ \int_0^t  e^{-2s C_p}<\chi_s,\phi><\chi_s,L_{\mu_s}\phi>ds \right\}\\
&& \quad -2C_p E \left\{ \int_0^t e^{-2s C_p}<\chi_s,\phi>^2ds \right\} \\
&& \qquad +  E\left\{\int_{\R^d}\int_0^t e^{-2s C_p} \left\|\nabla \phi \Gamma(\cdot,x)\chi_s\right\|^2 dx ds \right\}.
\end{eqnarray*}
Summing over a complete orthonormal system of $H_q$, one  gets
\begin{eqnarray*}
e^{-t C_p}E\left\{ \|\chi_t\|_{-q}^2 \right\} &= &   2\int_0^t e^{-2s C_p} E\left\{ <\chi_s,L_{\mu_s}^* \chi_s>_{-q} \right\} ds\\
&& \quad -2C_p E \left\{ \int_0^t e^{-2s C_p}\|\chi_s\|_{-q}^2ds \right\} \\
&& \qquad + E\left\{\int_{\R^d}\int_0^t e^{-2s C_p} \left\|\nabla^* \Gamma(\cdot,x)\chi_s\right\|^2 dx ds \right\}\\
&\le & 0,
\end{eqnarray*}
using the monotonicity condition \eqref{eq:mono}. Hence the result.
\end{proof}

\section{Disproof of Marchioro and Pulvirenti claim}  \label{app:disprove}

Denote by $x(t,y)$ the solution of the stochastic differential
equation
$$
d x_t = U_\epsilon(x_t,\mu_t)dt + \sigma dW_t,\qquad x_0=y.
$$

Recall that their mapping $\mathcal{S}$ is defined by
$$
(\mathcal{S}\mu)_t(A) =  \int P_t(A|y)\nu(dy)
$$
for all Borel sets $A\subset \R^2$, where $P_t(\cdot|y)$ are the
transition probabilities of the diffusion process $x(t,y)$.

Suppose that $\nu \in SM(m_1,m_2)$ and $(\mathcal{S}\mu)_t \in SM(m_1,m_2)$,
as claimed in \citet{Marchioro/Pulvirenti:1982}. It follows from Proposition
\ref{prop:jordan} that the Hahn-Jordan decomposition of $(\mathcal{S}\mu)_t$ is
$$
\left( (\mathcal{S}\mu)_t^1, (\mathcal{S}\mu)_t^2 \right)= \left(\int P_t(\cdot|y)\nu^1(dy), \int P_t(\cdot|y)\nu^2(dy)\right),
$$
if $(\nu^1,\nu^2)$ is the  Hahn-Jordan decomposition of $\nu$. Therefore,
there exist disjoint sets $A^1$ and $A^2$ so that
$(\mathcal{S}\mu)^\tau_t(A^\tau)=m_\tau$, $\tau=1,2$, $(\mathcal{S}\mu)^1_t(A^2)=0$ and $(\mathcal{S}\mu)^2_t(A^1)=0$.

It follows that there exist Borel sets $N^1$ and $N^2$ so that $\nu^\tau(N^\tau)=m_\tau$ and
$
P_t(A^\tau|y)=1$ for all $y\in N^\tau$, $\tau=1,2$. In addition $P_t(A^1|y)=0$ for  all $y\in N^2$, and
 $P_t(A^2|y)=0$ for  all $y\in N^1$.  Hence, $N^1$ and $N^2$ are disjoint, showing that $\nu^1(N^2)=0 = \nu^2(N^1)$.

Because for any $z\in\mathbb{R}^2$, $P_t(\cdot|z)$ has a positive density with respect to Lebesgue measure (since it can transformed into
a Wiener process with respect to an equivalent measure using Girsanov's formula), it follows that both $A^1$ and $A^2$ would have zero
Lebesgue measure, contradicting the equations $P_t(A^\tau|y)=1$ for all $y\in N^\tau$, $\tau=1,2$.

Therefore, $(\mathcal{S}\mu)_t \not \in SM(m_1,m_2)$.

\bibliographystyle{apalike}
\bibliography{d:/perso/mytex/bib/all2005}

\end{document}